\documentclass{birkjour}

\usepackage[bookmarks,
   pdfpagelabels=true, 
  ]{hyperref}

\usepackage{cite}

 \newtheorem{thm}{Theorem}[section]
 \newtheorem{cor}[thm]{Corollary}
 \newtheorem{lem}[thm]{Lemma}
 \newtheorem{prop}[thm]{Proposition}
 \theoremstyle{definition}
 
 \theoremstyle{remark}
 \newtheorem{rem}[thm]{Remark}
 \newtheorem{ex}[thm]{Example}
 \numberwithin{equation}{section}

  \newcommand{\textdef}[1]{\textit{#1}\index{#1}}
   
\newcommand{\NN}{{\mathbb N} }
\newcommand{\Rn}{{\mathbb R}^n }
\newcommand{\Sn}{{\mathbb S}^n }
\newcommand{\Snp}{{\mathbb S}^n_+ }
\newcommand{\Snpp}{{\mathbb S}^n_{++} }
\newcommand{\Snt}{{\mathbb S}^n_T }
\newcommand{\R}{{\mathbb R} }
\newcommand{\Tstar}{{T^{\star}} }
\newcommand{\SDP}{{\textbf SDP} }

\newcommand{\Bez}{{\rm Bez} }
\newcommand{\MM}{\mathcal{M} }
\newcommand{\TT}{\mathcal{T} }
\newcommand{\Ss}{\mathcal{S} }
\newcommand{\FF}{\mathcal{F} }
\newcommand{\A}{\mathcal{A} }

\DeclareMathOperator{\rank}{{rank}}
\DeclareMathOperator{\range}{{range}}
\DeclareMathOperator{\sd}{{sd}}

\begin{document}

%
%
%
%
%
%
%
%
%

\title[Maximum determinant positive definite Toeplitz
completions]{Maximum determinant positive definite\\ Toeplitz
completions}

\author[Sremac]{Stefan Sremac}\thanks{Research supported by The Natural Sciences and Engineering Research Council of Canada.}
\address{%
Department of Combinatorics and Optimization,\\
Faculty of Mathematics, University of Waterloo, \\
200 University Ave. W,\\
Waterloo, Ontario, \\
Canada N2L 3G1}
\email{ssremac@uwaterloo.ca}

\author[Woerdeman]{Hugo J. Woerdeman}\thanks{Research supported by Simons Foundation grant 355645.}
\address{Department of Mathematics\\ Drexel University \\ 
3141 Chestnut Street, \\
Philadelphia, PA 19104, \\
USA}
\email{hugo@math.drexel.edu}

\author[Wolkowicz]{Henry Wolkowicz}\thanks{Research supported by The Natural Sciences and Engineering Research Council of Canada.}
\address{%
Department of Combinatorics and Optimization,\\
Faculty of Mathematics, University of Waterloo, \\
200 University Ave. W,\\
Waterloo, Ontario, \\
Canada N2L 3G1}
\email{hwolkowi@uwaterloo.ca}


\date{\today}
\dedicatory{Dedicated to our friend Rien Kaashoek in celebration of his eightieth birthday.}

\begin{abstract}
We consider partial symmetric Toeplitz matrices where a positive definite
completion exists. We characterize those patterns where the maximum 
determinant completion is itself Toeplitz. We then extend these results with
positive definite replaced by positive semidefinite, and maximum
determinant replaced by maximum rank.
These results are used to determine the singularity degree of a family of
semidefinite optimization problems.
\end{abstract}

\maketitle
\section{Introduction}
In this paper we study the positive definite completion of a \textdef{partial
symmetric Toeplitz matrix, $\TT$}. 
The main contribution is Theorem \ref{thm:main}, where we present a characterization of 
those Toeplitz patterns for which the maximum determinant completion is 
Toeplitz, whenever the partial matrix is positive definite completable.
Part of this result answers a conjecture about the existence of a 
positive Toeplitz completion with a specific pattern.
A consequence of the main result is an extension to the maximum rank
completion in the positive
semidefinite case, and an application to the \emph{singularity degree}
of a family of \emph{semidefinite programs (SDPs)}.  In the following paragraphs we introduce relevant background information, state the main result, and motivate our pursuit. 
\index{$\TT$, partial symmetric Toeplitz matrix}

A \textdef{partial matrix} is a matrix in which some of the entries are assigned values while others are unspecified treated as variables.  For instance, 
\begin{equation}
\label{eq:M}
\MM:=\begin{bmatrix} 6 & 1 & x & 1 & 1 \cr 
1 & 6 & 1 & y & 1 \cr 
u & 1 & 6 & 1 & z \cr 
1 & v & 1 & 6 & 1 \cr 
1 & 1 & w & 1 & 6 \end{bmatrix}
\end{equation}
is a real partial matrix, where the unspecified entries are indicated by
letters.  A \textit{completion} of a partial matrix $\TT$ is obtained by
assigning values to the unspecified entries.  In other words, a matrix
$T$ (completely specified) is a completion of $\TT$ if it coincides with
$\TT$ over the specified entries: $T_{ij} = \TT_{ij}$, whenever $\TT_{ij}$ is
specified.  A \textdef{matrix completion problem} is to determine whether the partial matrix can be completed so as to satisfy a desired
property?  This type of problem has enjoyed considerable attention in
the literature due to applications in numerous areas,
e.g.,~\cite{MR2807419,Rechtparrilofazel}. For example
this is used in sensor network localization
\cite{kriswolk:09,laur:97b}, where the property is that the 
completion is a Euclidean
distance matrix with a given embedding dimension.
Related references for matrix completion problems
are e.g.,~\cite{MR1321785,MR2014037,MR2565240,MR2279160,MR1823516}.

The \text{pattern} of a partial matrix is the set of specified entries.
For example, the pattern of $\MM$ is all of the elements in diagonals
$-4,-3,-1,0,1,3,4$.  Whether a partial matrix is positive definite completable to some
property may depend on the values assigned to specified entries (the
data) and it may also depend on the pattern of specified entries.  A
question pursued throughout the literature is whether there exist
patterns admitting completions whenever the data satisfy some
assumptions.  Consider, for instance, the property of positive
definiteness.  A necessary condition for a partial matrix to have a
positive definite completion is that all completely specified
principal submatrices are positive definite.  We refer to such partial
matrices as \textdef{partially positive definite}.  Now we ask: what are
the patterns for which a positive definite completion exists whenever a
partial matrix having the pattern is partially positive definite?  In
~\cite{GrJoSaWo:84} the set of such patterns is shown to be fully
characterized by \emph{\textdef{chordal}ity of the graph} of the matrix.

In this work the desired property is \emph{symmetric Toeplitz positive
definite}. In particular, we consider the completion with maximum
determinant over all positive definite completions.  Recall that a real
symmetric $n\times n$ matrix $T$ is Toeplitz if there exist real numbers
$t_0,\dotso,t_{n-1}$ such that $T_{ij} = t_{\lvert i-j\rvert}$ for all
$i,j \in \{1,\dotso,n\}$.  A partial matrix is said to be
\textdef{partially symmetric Toeplitz} if the specified entries are symmetric and consist
of entire diagonals where the data is constant over each diagonal.  The
pattern of such a matrix indicates which diagonals are known and hence
is a subset of $\{0,\dotso,n-1\}$.  Here $0$ refers to the main
diagonal, $1$ refers to the super diagonal and so on.  The subdiagonals
need not be specified in the pattern since they are implied by symmetry.
In fact, since positive definite completions are trivial when the main
diagonal is not specified (and the determinant is unbounded), we assume
throughout that the main diagonal is specified. We therefore only consider 
patterns of increasing integers in the set $\{1,\dotso,n-1\}$.  The pattern of $\MM$, for instance, 
is $\{1,3,4\}$.  
\index{$T$, Toeplitz}
\index{Toeplitz, $T$}

For a partial matrix $\TT$ with pattern $P$ and $k \in P$, we let $t_k$ denote the value of $\TT$ on diagonal $k$ and we refer to $\{t_k : k \in P \cup \{0\} \}$
as the \textit{data} of $\TT$.  For $\MM$ the data is 
$\{t_0,t_1,t_3,t_4\}= \{6,1,1,1\}$.

We say that a partial Toeplitz matrix $\TT$ is \textdef{positive
(semi)definite completable} if there exists a positive (semi)definite completion of $\TT$.  In this case we denote by 
$\Tstar$, the unique positive definite completion of $\TT$ that maximizes the 
determinant over all positive definite completions.  We now state the
main contribution of this paper, a characterization of the Toeplitz patterns
where the maximum determinant completion is itself Toeplitz, whenever the
partial matrix is positive definite completable.
\index{maximum determinant positive definite completion, $\Tstar$}
\index{$\Tstar$, maximum determinant positive definite completion}
\begin{thm}
\label{thm:main}
Let $\emptyset \ne \textdef{$P \subseteq \{1,\dotso,n-1\}$}$ denote a
pattern of increasing integers.  The following are equivalent. 
\begin{enumerate}
\item 
\label{item:parttoep}
Let $\TT$ be a partial Toeplitz matrix have pattern P, and let $\TT$ 
be positive definite completable. Then $\Tstar$ is Toeplitz.
\item
\label{item:Prnk}
There exist $r,k \in \NN$ such that $P$ has one of the three forms:
\begin{itemize}
\item \textdef{$P_1 := \{k,2k,\dotso,rk\}$},
\item \textdef{$P_2 := \{k,2k,\dotso,(r-2)k,rk\}$, where $n = (r+1)k$},
\item \textdef{$P_3 := \{k,n-k\}$}.
\end{itemize}
\end{enumerate}
\end{thm}
The proof of Theorem \ref{thm:main} is presented in Section~\ref{sec:main}.  Note that for the 
partial Toeplitz matrix $\MM$ in \eqref{eq:M}, we can
set all the unspecified entries to $1$ and obtain a positive definite
completion. However, the maximum determinant completion is given, to four
decimal accuracy, when $x=z=u=w=0.3113$ and $y=v=0.4247$. But, 
this completion is \emph{not}
Toeplitz.  Indeed, the pattern of $\MM$ is not among the patterns of Theorem~\ref{thm:main}.

Positive definite Toeplitz matrices play an important role throughout
the mathematical sciences.  Correlation matrices of data arising from
time series,~\cite{MR941464}, and solutions to the trigonometric moment
problem,~\cite{MR1709182}, are two such examples.  
Among the early contributions to this area is the following sufficient
condition and characterization, for a special case of pattern $P_1$.
\begin{thm}[\!{\cite{DymGoh:81}}] 
\label{thm:banded}
If $\TT$ is a partially positive definite Toeplitz matrix with pattern
$P_1$ and $k=1$, then $\Tstar$ exists and is Toeplitz.
\end{thm}
\begin{thm}[\!{\cite[Theorem~1.1]{MR1709182}}]
\label{thm:jonlunnav}
A partially positive definite Toeplitz matrix is positive definite 
Toeplitz completable if, and only if, it has a pattern of the form $P_1$.
\end{thm}
In these two results the assumption on the partial matrix is that it is partially positive definite, whereas in Theorem~\ref{thm:main} we make the stronger assumption that a positive definite completion exists.  As a consequence, our characterization includes the patterns $P_2$ and $P_3$.  To the best of our knowledge pattern $P_2$ has not been addressed in the literature. 
A special case of pattern $P_3$,
with $k=1$, was considered in~\cite{MR1236734}, where the authors
characterize the data for which the pattern is positive definite completable.
In~\cite{MR2156414} the result is extended to arbitrary $k$ and
sufficient conditions for Toeplitz completions are provided.  Moreover,
the authors conjecture that whenever a partially positive definite
Toeplitz matrix with pattern $P_3$ is positive definite completable then it admits a
Toeplitz completion.  This conjecture is confirmed true in
Theorem~\ref{thm:main} and more specifically in Theorem~\ref{thm:main2}.

Our motivation for the maximum determinant completion comes from
optimization and the implications of the optimality conditions for
completion problems (see Theorem~\ref{thm:maxdet}).  In particular, a
positive definite completion problem may be formulated as an SDP.  
The \textit{central path} of standard interior point 
methods used to solve SDPs consists
of solutions to the maximum determinant problem. In the recent
work~\cite{SWW:17}, the maximum determinant problem is used to find
feasible points of SDPs when the usual regularity conditions are not
satisfied.  A consequence of Theorem~\ref{thm:main} is that when a
partially Toeplitz matrix having one of the patterns of the theorem
admits a positive semidefinite completion, but not a positive definite
one, then it has a maximum rank positive semidefinite completion that is 
Toeplitz. This result, as well as further discussion on the positive semidefinite case, are presented in 
Section~\ref{sec:sdp}. The application to finding the \emph{singularity degree}
of a family of  SDPs is presented in Section \ref{sec:sd}.

\section{Proof of Main Result with Consequences}
\label{sec:main}
To simplify the exposition, the proof of Theorem~\ref{thm:main} is
broken up into a series of results.  Throughout this section
we assume that every pattern $P$ is a non-empty subset of
$\{1,\dotso,n-1\}$, consisting of strictly increasing integers, and
$\TT$ denotes an $n\times n$ partial symmetric Toeplitz matrix with
pattern (or form) $P$.  We begin by presenting the optimality conditions for the maximum determinant problem.  
\index{pattern, $P$}
\index{$P$, pattern}
\begin{thm}
\label{thm:maxdet}
Let $\TT$ be of the form $P$ and positive definite completable. Then 
$\Tstar$ exists, is unique, and satisfies $(\Tstar)^{-1}_{i,j} = 0$,
whenever $\lvert i - j \rvert \notin P$.
\end{thm}  
\begin{proof}  
This result is proved for general positive definite completions 
in~\cite{GrJoSaWo:84}. See also~\cite{SWW:17}.
\end{proof}  
For general positive definite completion problems, this result simply
states that the inverse of the completion of maximum determinant has
zeros in the unspecified (or free) entries.  Since we are interested in
Toeplitz completions, we may say something further using a permutation
under which Toeplitz matrices are invariant.  Let $K$ be the symmetric
$n\times n$ anti-diagonal matrix defined as:
\index{$K$, anti-diagonal permutation}
\index{anti-diagonal permutation, $K$}
\begin{equation}
\label{eq:antidiag}
K_{ij} := \begin{cases}
1 \quad \text{if } i+j = n+1, \\
0 \quad \text{otherwise},
\end{cases}
\end{equation}
i.e.,~$K$ is the permutation matrix that reverses the order of the
sequence $\{1,2,\ldots,n\}$.
\begin{lem}
\label{lem:K}
Let $\TT$ be of the form $P$ and positive definite completable. 
Let $\Tstar$ be the maximum determinant completion, and let
$K$ be the anti-diagonal permutation matrix in \eqref{eq:antidiag}.
Then the following hold.
\begin{enumerate}
\item 
\label{item:KTK}
 $\Tstar = K\Tstar K$.
\item 
\label{item:PTstar}
If $P$ is of the form $P_2$ with $k=1$, i.e.,~$P = \{1,2,\dotso,n-3,n-1\}$, then $\Tstar$ is Toeplitz.
\end{enumerate}
\end{lem}
\begin{proof}
For Item \ref{item:KTK}, it is a simple exercise to verify that the
permutation reverses the order of the rows and columns and we have 
\[
[K\Tstar K]_{ij} = \Tstar_{n+1-i,n+1-j}, \,
\forall i,j \in \{1,\dotso,n\}.
\]
Moreover,
\[
\lvert n+1-i - (n+1-j)\rvert = \lvert i-j\rvert.
\]
Therefore, it follows that
\[
[K\Tstar K]_{ij} = \Tstar_{n+1-i,n+1-j} = \Tstar_{ij} = t_{\lvert i-j\rvert},
\,\, \forall \lvert i-j \rvert \in P\cup \{0\}. 
\]
Hence $K\Tstar K$ is a completion of $\TT$.  Moreover, $K\cdot K$ is
an automorphism of the cone of positive definite matrices. Hence
$K\Tstar K$ is a positive definite completion of $\TT$, and since $K$ is
a permutation matrix, we conclude that $\det(K\Tstar K) = \det(\Tstar)$.
By Theorem~\ref{thm:maxdet}, $\Tstar$ is the unique maximizer of the
determinant. Therefore $\Tstar = K\Tstar K$, as desired.

For Item \ref{item:PTstar}, we let $\TT$ be as in the hypothesis and
note that the only unspecified entries are $(1,n-1)$ and $(2,n)$, and
their symmetric counterparts.  Therefore it suffices to show that
$\Tstar_{1,n-1} = \Tstar_{2,n}$.  By applying Item \ref{item:KTK} we get
\[
\Tstar_{1,n-1} = [K\Tstar K]_{1,n-1} = \Tstar_{n+1-1,n+1-(n-1)} = \Tstar_{n,2} = \Tstar_{2,n},
\]
as desired.
\end{proof}

The pattern $\{1,2,\dotso,n-3,n-1\}$ in Lemma \ref{lem:K}, above,
is a special case of pattern $P_2$ with $k=1$.  In fact, we show that a
general pattern $P_2$ may always be reduced to this special case.  A
further observation is that this specific pattern is nearly of the form
$P_1$.  Indeed, if the diagonal $n-2$ were specified, the pattern would be of
the form $P_1$.  In fact, for any pattern of the form $P_2$, if the
diagonal $(r+1)k$ were specified, the pattern would be of the form
$P_1$.  We now state a useful lemma for proving that 
Theorem \ref{thm:main}, Item \ref{item:Prnk} implies 
Theorem \ref{thm:main}, Item \ref{item:parttoep}, when $P$ is of the 
form $P_1$ or $P_2$.

\begin{lem}
\label{lem:block}
Let $\Ss$ be a partial $n\times n$ positive definite completable symmetric matrix and $Q$ a permutation matrix of order $n$ such that
\[
Q^T\Ss Q = \begin{bmatrix}
\Ss_1 & &&  \\
& \Ss_2 & &  \\
& & \ddots &  \\
& && \Ss_{\ell} \\
\end{bmatrix},
\]
for some $\ell \in \NN$. Here each block $\Ss_i$ is a 
partial symmetric matrix for
$i \in \{1,\dotso,\ell\}$, and the elements outside of the blocks are
all unspecified.  Then the maximum determinant completion of $\Ss_i$,
denoted $S^{\star}_i$, exists and is unique.  Moreover, the unique
maximum determinant completion of $\Ss$ is given by
\[
S^{\star} = Q \begin{bmatrix}
S^{\star}_1 &0 &\cdots&0  \\
0 & S^{\star}_2 &\cdots &0  \\
\vdots &\vdots & \ddots & \vdots \\
0& 0&\cdots& S^{\star}_{\ell} \\
\end{bmatrix}Q^T.
\]
\end{lem}
\begin{proof}
Since $Q^T\cdot Q$ is an automorphism of the positive definite matrices,
with inverse $Q\cdot Q^T$, we have that $Q^T\Ss Q$ is positive definite completable and
admits a unique maximum determinant completion, say $\hat{S}$.
Moreover, under the map $Q\cdot Q^T$, every completion of $Q^T \Ss Q$
corresponds to a unique completion of $\Ss$, with the same determinant,
since the determinant is invariant under the transformation $Q\cdot
Q^T$.  Therefore, we have $S^{\star} = Q\hat{S} Q^T$.  Now we show that
$\hat{S}$ has the block diagonal form.  Observe that $\Ss_i$ is
positive definite completable, take for instance the positive definite submatrices of
$\hat{S}$ corresponding to the blocks $\Ss_i$.  Thus $S^{\star}_i$ is
well defined, and by the determinant Fischer inequality, 
e.g.,~\cite[Theorem 7.8.3]{HoJo:85}, we  have
\[
\hat{S} = \begin{bmatrix}
S^{\star}_1 &0 &\cdots&0  \\
0 & S^{\star}_2 &\cdots &0  \\
\vdots &\vdots & \ddots & \vdots \\
0& 0&\cdots& S^{\star}_{\ell} \\
\end{bmatrix},
\] 
as desired.
\end{proof}

In~\cite{MR1709182} it is shown that a partial Toeplitz matrix of the
form $P_1$ with $rk = n-1$ can be permuted into a block diagonal matrix as in 
Lemma~\ref{lem:block}.  We use this observation and extend it to
all patterns of the form $P_1$, as well as patterns of the form $P_2$, in the following.

\begin{prop}
\label{prop:P1P2}
Let $\TT$ be positive definite completable and of the form $P_1$ or $P_2$. Then $\Tstar$ is Toeplitz.
\end{prop}
\begin{proof}
Let $\TT$ be of the form $P_1$ with data $\{t_0,t_k, t_{2k},\dotso,t_{rk}\}$ and let $p \ge r$ be the largest integer so that $pk \le n-1$.  As in~\cite{MR1709182}, there exists a permutation matrix $Q$ of order $n$ such that 
\[
Q^T\TT Q = \begin{bmatrix}
\TT_0 & &&&&  \\
& \ddots & &&&  \\
& & \TT_0 &&&  \\
& && \TT_1 &&\\
&&&&\ddots&\\
&&&&& \TT_1
\end{bmatrix},
\]
where $\TT_0$ is a $(p+1)\times (p+1)$ partial Toeplitz matrix occuring
$n-pk$ times and and $\TT_1$ is a $p\times p$ partial Toeplitz.
Moreover, $\TT_0$ and $\TT_1$ are both partially positive definite.  Let
us first consider the case $p=r$.  Then $\TT_0$ and $\TT_1$ are actually
fully specified, and the maximum determinant completion of $Q^T \TT Q$,
as in Lemma~\ref{lem:block}, is obtained by fixing the elements outside
of the blocks to $0$.  After permuting back to the original form,
$\Tstar$ has zeros in every unspecified entry. Hence it is Toeplitz.
Now suppose $p > r$.  Then $\TT_0$ is a partial Toeplitz matrix with
pattern $\{1,2,\dotso,r\}$ and data $\{t_0,t_k,t_{2k},\dotso,t_{rk}\}$
and $\TT_1$ is a partial Toeplitz matrix having the same pattern and
data as $\TT_0$, but one dimension smaller.  That is, $\TT_1$ is a
partial principal submatrix of $\TT_0$.  By Theorem~\ref{thm:banded}
both $\TT_0$ and $\TT_1$ are positive definite completable and their
maximum determinant completions, $\TT_0^{\star}$ and $\TT_1^{\star}$,
are Toeplitz.  Let $\{a_{(r+1)k}, a_{(r+2)k}, \dotso, a_{pk}\}$ be the
data of $\TT_0^{\star}$ corresponding to the unspecified entries and let
$\{b_{(r+1)k}, b_{(r+2)k}, \dotso, b_{(p-1)k}\}$, be the data
corresponding to the unspecified entries of $\TT_1$.  By the permanence
principle of \cite{MR892133}, $\TT_1$ is a principle submatrix of
$\TT_0$ and therefore $b_i = a_i$, for all 
$i \in \{(r+1)k,(r+2)k,\dotso,(p-1)k\}$.
By Lemma~\ref{lem:block}, the maximum determinant completion of $Q^T\TT Q$ is obtained by completing $\TT_0$ and $\TT_1$ to $\TT_0^{\star}$ and $\TT_1^{\star}$ respectively, and setting the entries outside of the blocks to zero.  After permuting back to the original form we get that $\Tstar$ is Toeplitz with data $a_{(r+1)k}, a_{(r+2)k}, \dotso, a_{pk}$ in the diagonals $(r+1)k, (r+2)k, \dotso,pk$ and zeros in all other unspecified diagonals.

Now suppose that $\TT$ is of the form $P_2$.  By applying the same
permutation as above, and by using the fact that $n=(r+1)k$ and each
block $\TT_0$ is of size $r+1$, we see that the submatrix consisting only of blocks $\TT_0$ is of size
\[
(n-rk)(r+1) = ((r+1)k-rk)(r+1) = k(r+1) = n.
\]
Hence, 
\[
Q^T\TT Q = \begin{bmatrix}
\TT_0 & &  \\
& \ddots &   \\
& & \TT_0   \\
\end{bmatrix},
\]
where $\TT_0$ is a partial matrix with pattern $\{1,2,\dotso,r-2,r\}$ and data
\[
\{t_0,t_k,t_{2k},\dotso,t_{(r-2)k},t_{rk}\}.
\]  
The unspecified elements of diagonal $(r-1)k$ of $\TT$ are contained in
the unspecified elements of diagonal $r-1$ of the partial matrices
$\TT_0$.  By Lemma~\ref{lem:K}, the maximum determinant completion of $\TT_0$ is Toeplitz with value $t_{(r-1)k}$ in the unspecified diagonal.  As in the above, after completing $Q^T\TT Q$ to its maximum determinant positive definite completion and permuting back to the original form, we obtain the maximum determinant Toeplitz completion of $\TT$ with value $t_{(r-1)k}$ in the diagonal $(r-1)k$ and zeros in every other unspecified diagonal, as desired.
\end{proof}

We now turn our attention to patterns of the form $P_3$.  Let $J$ denote
the $n\times n$ lower triangular Jordan block with eigenvalue $0$.  That
is, $J$ has ones on diagonal $-1$ and zeros everywhere else. We also let
$e_1,\dotso,e_n \in \Rn$ denote the columns of the identity matrix,
i.e.,~the canonical unit vectors.  With this notation we have $J = \sum_{j=1}^{n-1} e_{j+1}e_j^T$.  We state several technical results regarding $J$ in the following lemma.
\index{$J$, Jordan block}
\index{Jordan block, $J$}

\begin{lem}
\label{J} 
With $J$ defined as above and $k,l \in \{0,1,\dotso,n-1\}$, the following hold.
\begin{enumerate}
\item 
\label{item:sumes}
$J^k = \sum_{j=1}^{n-k} e_{j+k}e_j^T$.
\item 
\label{item:nnzJ}
$J^k(J^T)^l$ has nonzero elements only in the diagonal $l-k$.
\item
\label{item:JJz}
 If $k < l$, then 
\[
J^k(J^T)^l - J^{n-l} (J^T)^{n-k} = 0 \ \iff \ l=n-k.
\] 
\end{enumerate}
\end{lem}
\begin{proof} 
For Item \ref{item:sumes} the result clearly holds when $k\in \{0,1\}$.
Now observe that for integers of suitable size
$(e_ke_l^T)(e_ie_j^T) \ne 0$ if, and only if, $l=i$ in which case the product is $e_ke_j^T$.  Thus we have
\[
J^2 = \left( \sum_{j=1}^{n-1} e_{j+1}e_j^T\right)\left(\sum_{j=1}^{n-1} e_{j+1}e_j^T\right) = \sum_{j=2}^{n-1} (e_{j+1}e_j^T)(e_je_{j-1}^T) = \sum_{j=1}^{n-2} e_{j+2}e_j^T.
\]
Applying an induction argument yields the desired expression for arbitrary $k$.

For Item \ref{item:nnzJ}, we use the result of item  \ref{item:sumes} to get
\begin{align*}
J^k(J^T)^l &= \left( \sum_{j=1}^{n-k} e_{j+k}e_j^T\right)\left(\sum_{j=1}^{n-l} e_je_{j+l}^T\right), \\
&= \sum_{j=1}^{n-\max \ \{k,l\}} (e_{j+k}e_j^T)(e_je_{j+l}^T), \\
&= \sum_{j=1}^{n-\max \ \{k,l\}} e_{j+k}e_{j+l}^T.
\end{align*}
The nonzero elements of this matrix are contained in the diagonal $j+l - (j+k) = l-k$.

Finally, for Item \ref{item:JJz} we have 
\[
J^k(J^T)^l - J^{n-l} (J^T)^{n-k}  = \sum_{j=1}^{n-l} e_{j+k}e_{j+l}^T - \sum_{j=1}^k e_{j+n-l}e_{j+n-k}^T.
\] 
This matrix is the zero matrix if, and only if, $l=n-k$. 
\end{proof}
We now state a special case of the Schur-Cohn Criterion using the matrix
$J$.  We let $\Sn$ denote the Euclidean space of symmetric matrices,
$\Snpp$ the cone of positive definite matrices, and $\Snt$ the subset of
symmetric, positive definite, Toeplitz matrices.  
\begin{thm}[Schur-Cohn Criterion,~\cite{MR638124}]
\label{thm:schurcohn}
Let $f(z) = a_0 + a_1z + \cdots + a_nz^n$ be a polynomial with real
coefficients $a:=(a_0,\dotso,a_n)$. Let
\[
A(a) := \sum_{j=0}^{n-1} a_jJ^j, \ B(a):= \sum_{j=1}^{n-1} a_j \left( J^T\right)^{n-j}.
\]
Then every root of $f(z)$ satisfies $\lvert z \rvert > 1$ if, and only
if, 
\[
\Bez(a) := A(a)A(a)^T - B(a)^TB(a) \in \Snpp,
\]
where the matrix $\Bez(a)$ is the Toeplitz Bezoutian.
Moreover $\Bez(a)^{-1}$ is Toeplitz.
\end{thm}
The Schur-Cohn criterion is usually stated for the case where the roots
are contained within the interior of the unit disk, but a simple
reversal of the coefficients, as described in Chapter~X of
\cite{MR0225972}, leads to the above statement. For further information
on the \textit{Toeplitz Bezoutian}  $\Bez(a)$, and for a proof of the fact 
that $\Bez(a)^{-1}$ is
Toeplitz, see  e.g.,~\cite{MR2656823}. 

We now present a result on the 
maximum determinant completion of partial Toeplitz matrices with pattern $P_3$.
\begin{prop}
\label{prop:P3} 
Let $\TT$ be positive definite completable and of the form $P_3$. Then $\Tstar$ is Toeplitz.
\end{prop}

\begin{proof}
Let $\TT$ be as in the hypothesis with pattern of the form $P_3$ defined the integer $k$.  Furthermore, let ${\mathcal O} \subset {\mathbb R}_{++} \times
{\mathbb R}^2$ consist of all triples $(t_0,t_k,t_{n-k})$ so that the
partial Toeplitz matrix with pattern $P_3$ and data
$\{t_0,t_k,t_{n-k}\}$ is positive definite completable. Then it can be verified that ${\mathcal O}$ is an open convex set, and thus in particular connected. We let ${\mathcal U} \subseteq {\mathcal O}$ consist of those triples $(t_0,t_k,t_{n-k})$ for which the corresponding maximum determinant completion is Toeplitz and we claim that ${\mathcal U} = {\mathcal O}$. Clearly ${\mathcal U} \neq \emptyset$ as $(t_k,0,0) \in {\mathcal U}$ for all $t_k>0$. We show that ${\mathcal U}$ is both open and closed in ${\mathcal O}$, which together with the connectedness of ${\mathcal O}$ yields that ${\mathcal U} = {\mathcal O}$.

First observe that the map $F: {\mathcal O} \to \Snpp$ that takes
$(t_0,t_k,t_{n-k})$ to its corresponding positive definite maximum
determinant completion is continuous; see, for instance,
\cite{MR1824072}. Next, the Toeplitz positive definite matrices, 
$\Snpp \cap \Snt$, form a closed subset of $\Snpp$ since $\Snt$ is closed. Thus ${\mathcal U}
= F^{-1} (\Snpp\cap \Snt)$ is closed in $\mathcal{O}$. 

To show that ${\mathcal U}$ is also open, we introduce the set, 
\[
{\mathcal P} := \{ (p,q,r) \in {\mathbb R}_{++} \times {\mathbb R}^2 : p+qz^k+rz^{n-k} {\rm \ has \ all \ roots \ satisfy \ } |z|>1 \} .
\] 
Since the region $\lvert z \rvert > 1$ is an open subset of the complex plane, $\mathcal P$ is an open set.  We consider the map $G: {\mathcal P} \to \R^3$ defined as
\[
G(p,q,r) = ( [\Bez(p,q,r)^{-1}]_{11},  [\Bez(p,q,r)^{-1}]_{k1},  [\Bez(p,q,r)^{-1}]_{n-k,1} ),
\] 
where by abuse of notation $\Bez(p,q,r)$ is the Toeplitz Bezoutian of Theorem~\ref{thm:schurcohn}:
\begin{align*}
\Bez(p,q,r) &= (pJ^0+qJ^k+rJ^{n-k})(pJ^0+qJ^k+rJ^{n-k})^T \\
& \qquad - (rJ^k+qJ^{n-k})(rJ^k+qJ^{n-k})^T.
\end{align*}
Then $G$ is continuous and we show that its image is exactly $\mathcal{U}$.  By Theorem~\ref{thm:schurcohn}, for any $(p,q,r) \in \mathcal{P}$ we have
\[
\Bez(p,q,r) \in \Snpp, \ \Bez(p,q,r)^{-1} \in \Snpp \cap \Snt.
\]  
Thus $\Bez(p,q,r)^{-1}$ is a completion of the partial matrix having pattern $P_3$ and data $\{[\Bez(p,q,r)^{-1}]_{11},  [\Bez(p,q,r)^{-1}]_{k1},  [\Bez(p,q,r)^{-1}]_{n-k,1}\}$.  It follows that $G(\mathcal{P}) \subseteq \mathcal{O}$.  Moreover, expanding $\Bez(p,q,r)$ we obtain diagonal terms as well as terms of the form $J^0(J^T)^k$ and $J^0(J^T)^{n-k}$, where the coefficients have been omitted.  By Lemma~\ref{J}, $\Bez(p,q,r)$ has non-zero values only in entries of the diagonals $0,k,n-k$.  Note that the term $J^k(J^T)^{n-k}$ cancels out in the expansion.  Thus by Theorem~\ref{thm:maxdet}, $\Bez(p,q,r)^{-1}$ is a maximum determinant completion of the partial matrix with pattern $P_3$ and data $\{[\Bez(p,q,r)^{-1}]_{11},  [\Bez(p,q,r)^{-1}]_{k1},  [\Bez(p,q,r)^{-1}]_{n-k,1}\}$ and $G(\mathcal{P}) \subseteq \mathcal{U}$.  To show equality, let $(t_0,t_k,t_{n-k}) \in \mathcal{U}$ and let $F(t_0,t_k,t_{n-k})$, as above, be the maximum determinant completion of the partial matrix with pattern $P_3$ and data $\{t_0,t_k,t_{n-k}\}$ which is Toeplitz.  Let $f_0$, $f_k$, and $f_{n-k}$ be the $(1,1)$, $(k+1,1)$ and $(n-k+1,1)$ elements of $F(t_0,t_kt_{n-k})^{-1}$ respectively.  Then by the Gohberg-Semencul formula for the inversion of a symmetric Toeplitz matrix (see \cite{MR0353038,MR1038316}) we have
\begin{align*}
F(t_0,t_kt_{n-k})^{-1} &= \frac{1}{f_0}(f_0J^0+f_kJ^k+f_{n-k}J^{n-k})(f_0J^0+f_kJ^k+f_{n-k}J^{n-k})^T \\
& \qquad -  \frac{1}{f_0}(f_{n-k}J^k+f_kJ^{n-k})(f_{n-k}J^k+f_kJ^{n-k})^T, \\
&= \Bez \left(\sqrt{f_0}, \frac{f_k}{\sqrt{f_0}}, \frac{f_{n-k}}{\sqrt{f_0}}\right).
\end{align*}
Since $F(t_0,t_kt_{n-k})^{-1} \in \Snpp$, it follows that $\left(\sqrt{f_0}, \frac{f_k}{\sqrt{f_0}}, \frac{f_{n-k}}{\sqrt{f_0}}\right) \in \mathcal{P}$ and 
\[
G\left(\sqrt{f_0}, \frac{f_k}{\sqrt{f_0}}, \frac{f_{n-k}}{\sqrt{f_0}}\right) = (t_0,t_k,t_{n-k}).
\]
Therefore $G(\mathcal{P}) = \mathcal{U}$.  Moreover, from the above we have that
\[
G^{-1}(t_0,t_k,t_{n-k}) = \left(\sqrt{f_0}, \frac{f_k}{\sqrt{f_0}}, \frac{f_{n-k}}{\sqrt{f_0}}\right),
\]
with $f_0,f_k$, and $f_{n-k}$ defined above.  Since $G^{-1}$ is continuous, $G^{-1}(\mathcal{U}) = \mathcal{P}$, and $\mathcal{P}$ is an open set, we conclude that $\mathcal{U}$ is open, as desired.
\end{proof}

Now we are ready to prove Theorem~\ref{thm:main}.
\begin{proof}[Proof of Theorem~\ref{thm:main}]
The direction $(\ref{item:Prnk})\!\! \implies\!\! (\ref{item:parttoep})$ 
follows from Proposition~\ref{prop:P1P2} and Proposition~\ref{prop:P3}.

For the direction $(\ref{item:parttoep})\!\! \implies\!\!  (\ref{item:Prnk})$,
let $\TT$ be positive definite completable with pattern $P = \{k_1,\dotso,k_s\}$, $k_0 = 0$, and data $\{t_0,t_1,\dotso,t_s\}$. Assume there exists 
data $t_j$ for the diagonal $k_j$, $j \in \{0,\dotso,s\}$, and that $\Tstar$ is Toeplitz.  Then by Theorem~\ref{thm:maxdet}, $(\Tstar)^{-1}$ has nonzero entries only in the diagonals $P\cup \{0\}$ (and their symmetric counterparts).  We denote by $a_j$ the value of the first column of $(\Tstar)^{-1}$ in the row $k_j+1$ for all $j \in \{0,\dotso,s\}$, and define
\[
A := \sum_{j=0}^s a_j J^{k_j}, \ B:= \sum_{j=1}^s a_j \left( J^T\right)^{n-k_j}.
\]
The Gohberg-Semencul formula gives us that $T^{-1} = \frac{1}{a_0}(AA^T
- B^TB)$.  Substituting in the expressions for $A$ and $B$ and
expanding, we obtain $(\Tstar)^{-1}$ is a linear combination of the
following types of terms, along with their symmetric counterparts:
\[
J^{k_j}(J^T)^{k_j}, \quad  J^{k_0}(J^T)^{k_j}, \quad J^{k_j}(J^T)^{k_l} - J^{n-k_j}(J^T)^{n-k_l}, \ j \ne l.
\]
By Lemma~\ref{J}, the first type of term has nonzero entries only on the
main diagonal, and the second type of term has nonzero entries only on
the diagonals belonging to $P$.  The third type of term has nonzero
entries only on the diagonals $\pm \lvert k_j-k_l\rvert$.  As we have
already observed in the proof of Proposition~\ref{prop:P3}, the set of
data for which $\TT$ is positive definite completable is an open set. 
We may therefore perturb the data of $\TT$ so that the entries 
$a_0,\dotso,a_j$ of the inverse do not all lie on the same proper linear 
manifold. Then terms of the form $J^{k_j}(J^T)^{k_l} -
J^{n-k_j}(J^T)^{n-k_l}$ with  $j \ne l$ do not cancel each other out. We
conclude that, for each pair $j < l$, we have $k_j - k_l \in P$ or $J^{k_j}(J^T)^{k_l} - J^{n-k_j}(J^T)^{n-k_l} = 0$.  By Lemma~\ref{J} the second alternative is equivalent to $l = n-j$.  Using this observation we now proceed to show that $P$ has one of the specified forms.  

Let $1\le r \le s$ be the largest integer such that $\{k_1,\dotso,k_r\}$
is of the form $P_1$, i.e.,~$k_2 = 2k_1$, $k_3 = 3k_1$, etc\ldots.  If
$r=s$, then we are done. Therefore we may assume $s \ge r+1$. Now we show 
that in fact $s=r+1$. We have that $k_{r+1} - k_1 \in P$ or $k_{r+1} = n-k_1$.  
We show that the first case does not hold. Indeed if $k_{r+1} - k_1 \in
P$, then it follows that $k_{r+1} - k_1 \in \{k_1,\dotso,k_r\}$. This
implies that
\[
k_{r+1} \in \{2k_1,\dotso,rk_1, (r+1)k_1\} = \{k_2,\dotso,k_r,(r+1)k_1\}.
\]
Clearly $k_{r+1} \notin \{k_2,\dotso,k_r\}$, and if $k_{r+1} =
(r+1)k_1$, then $r$ is not maximal, a contradiction.  Therefore $k_{r+1}
= n-k_1$.  To show that $s=r+1$, suppose to the contrary that $s \ge
r+2$.  Then $k_{r+2} - k_1 \in P$ or $k_{r+2} = n-k_1$.  The latter does
not hold since then $k_{r+2} = k_{r+1}$.  Thus we have $k_{r+2} - k_1
\in \{k_1,\dotso,k_r, k_{r+1}\}$, which implies that
\[
k_{r+2} \in \{2k_1,\dotso,rk_1,(r+1)k_1,k_{r+1}+k_1\} = \{k_2,\dotso,k_r,k_r+k_1,n\}.
\] 
Since $k_{r+2} \notin \{k_2,\dotso,k_r,n\}$, we have $k_{r+2} =
k_r+k_1$. Therefore, since $k_r < k_{r+1}<k_{r+2}$, we have that $0<
k_{r+2} - k_{r+1} <k_1$, and moreover, $k_{r+2} - k_{r+1} \notin P$.  It follows that $k_{r+2} = n-k_{r+1} = k_1$, a contradiction.  

We have shown that $P = \{k_1,2k_1,\dotso,rk_1,k_s\}$ with $k_s =
n-k_1$. If $r=1$, then $P$ is of the form $P_3$.  On the other hand if
$r \ge 2$, then we observe that $\{k_s -k_r, \dotso,k_s - k_2\} \subseteq P$, or equivalently,
\[
\{k_s -k_r, \dotso,k_s - k_2\} \subseteq \{k_2,\dotso,k_r\}.
\] 
Since the above sets of identical cardinality, distinct increasing elements, 
we conclude that $k_s - k_2 = k_r$.  Rearranging, we obtain that $k_s = (r+2)k_1$ and $P$ is of the form $P_2$, as desired.
\end{proof}

\begin{rem}
The results of this section have been stated for the symmetric real case for 
simplicity and for application to \SDP in the following Section \ref{sec:sdp}.
With obvious modifications, our results extend to the Hermitian case.
\end{rem}

\section{Semidefinite Toeplitz Completions}
\label{sec:sdp}
In this section we extend the results of Theorem~\ref{thm:main} to
positive semidefinite completions.  In the case where all completions
are singular, the maximum determinant is not useful for identifying a
Toeplitz one, however, a recent result of~\cite{SWW:17} allows us to extend our observations to the semidefinite case.  Given a partial symmetric Toeplitz matrix, $\TT$, a positive semidefinite completion of $\TT$ may be obtained by solving an SDP feasibility problem.  Indeed, if $\TT$ has pattern $P$ and data $\{t_k : k \in P\cup \{0\}\}$, then the positive semidefinite completions of $\TT$ are exactly the set 
\begin{equation}
\label{eq:feasibleset}
\FF := \{X \in \Snp : \A(X) = b \},
\end{equation}
where $\A$ is a linear map and $b$ a real vector in the image space of $\A$ satisfying 
\[
[\A(X)]_{ik} = \langle E_{i,i+k}, X\rangle, \ b_{ik} = t_k, \quad i \in \{1,2,\dotso,n-k\}, \ k \in P\cup \{0\}.
\]
Here $E_{i,j}$ is the symmetric matrix having a one in the entries $(i,j)$ and $(j,i)$ and zeros everywhere else and we use the trace inner product: $\langle X,Y \rangle  = {\rm tr} ( XY )$.  The maximum determinant is used extensively in SDP, for example, the central path of interior point methods is defined by solutions to the maximum determinant problem.  If $\FF$ is nonempty but does not contain a positive definite matrix, the maximum determinant may still be applied by perturbing $\FF$ so that it does intersect the set of positive definite matrices.  Consider the following parametric optimization problem  
\[
X(\alpha) := \arg \max \ \{ \det(X) : X \in \FF(\alpha)\},
\]
where $\FF(\alpha) := \{X \in \Snp : \A(X) = b + \alpha \A(I)\}$ and $\alpha > 0$.  For each $\alpha > 0$, the solution $X(\alpha)$ is contained in the relative interior of $\FF(\alpha)$.  It is somewhat intuitive that if the limit of these solutions is taken as $\alpha$ decreases to $0$, we should obtain an element of the relative interior of $\FF(0) = \FF$.  Indeed, the following result confirms this intuition.  We denote by $\A^*$ the adjoint of $\A$.
\begin{thm}
\label{thm:Xalpha}
Let $\FF \ne \emptyset$ and $X(\alpha)$ be as above.  Then there exists $\bar{X}$ in the relative interior of $\FF$ such that $\lim_{\alpha \searrow 0}X(\alpha) = \bar{X}$. Moreover, $\bar{Z} := \lim_{\alpha \searrow 0}\alpha (X(\alpha))^{-1}$ exists and satisfies $\bar{X} \bar{Z} = 0$ and $\bar{Z} \in \range(\A^*)$.
\end{thm}
\begin{proof}
See Section~3 of~\cite{SWW:17}.
\end{proof}
An immediate consequence of this result is the following.
\begin{cor}
\label{cor:sdp}
Let $\TT$ be an $n\times n$ partial symmetric Toeplitz matrix of the form $P_1$, $P_2$, or $P_3$.  If $\TT$ admits a positive semidefinite completion then it admits a maximum rank completion that is Toeplitz.
\end{cor}
\begin{proof}
Let $\TT$ be as in the hypothesis with data $\{t_0,t_1, \dotso, t_s\}$ and let $\FF$ be the set of positive semidefinite completions, as above.  If $\FF \cap \Snpp \ne \emptyset$, then the maximum determinant completion is Toeplitz by Theorem~\ref{thm:main} and is of maximum rank.  Now suppose $\FF \in \Snp \setminus \Snpp$ and observe that for every $\alpha > 0$, $\FF(\alpha)$ consists of solutions to to the completion problem having pattern $P_1$, $P_2$, or $P_3$ with data  $\{t_0+\alpha,t_1, \dotso, t_s\}$ and there exists a positive definite completion.  Thus $X(\alpha)$ is Toeplitz for each $\alpha > 0$ and since the Toeplitz matrices are closed, the limit point $\bar{X}$, of Theorem~\ref{thm:Xalpha}, is Toeplitz.  The relative interior of $\FF$ corresponds to those matrices having maximum rank over all of $\FF$, hence $\bar{X}$ has maximum rank, as desired.
\end{proof}
\begin{rem}
In Theorem~2.2 of \cite{MR1444091} the author gives in the case of two prescribed diagonals (in the strict lower triangular part) necessary and sufficient conditions on the data for the existence of a Toeplitz positive semidefinite completion. In Theorem~10 of \cite{MR2156414} the authors give in the case of pattern $P_3$ necessary and sufficient conditions for the existence of a positive semidefinite completion. If one is able to verify that the conditions are the same, which will require some tenacity, then one would have an alternative proof that for the pattern $P_3$ positive semidefinite completability implies the existence of a Toeplitz positive semidefinite completion. Their results are all stated for the real case, so one advantage of the approach here is that it readily generalizes to the complex Hermitian case.
\end{rem}
While Theorem~\ref{thm:main} characterizes patterns for which the maximum determinant completion is automatically Toeplitz and Corollary~\ref{cor:sdp} addresses the maximum rank completions, one may merely be interested in the existence of a Toeplitz completion when a positive semidefinite one exists. Obviously, the patterns in Theorem \ref{thm:main} fall in this category, but as we see in the following result, there are more. 

\begin{thm}
\label{thm:main2}
Define the patterns
\begin{itemize}
\item $P_2' := \{k,2k,\dotso,(r-2)k,rk\}$,  
\item $P_3' := \{k,r\}$ where $n \ge k+r$.
\end{itemize}
If $\TT$ is an $n\times n$ positive semidefinite completable partial Toeplitz matrix with a pattern in the set $\{P_1,P_2',P_3'\}$, then $\TT$ has a Toeplitz positive semidefinite completion.
\end{thm}
\begin{proof}
For pattern $P_1$ this is a consequence of Corollary~\ref{cor:sdp}.
Note that $P_2'$ and $P_3'$ are obtained from $P_2$ and $P_3$,
respectively, by relaxing the restriction on $n$, i.e.,~allowing $n$ to be larger.  Using the results we already have for $P_2$ and $P_3$ we fill in some of the diagonals of $\TT$ to obtain a new partial matrix of the form $P_1$.  We show the proof only for patterns of the form $P_2'$ since the same approach may be used for patterns of the form $P_3'$.

Suppose $\TT$ has pattern $P_2'$ and consider the partial submatrix containing the first $(r+1)k$ rows and columns.  This partial matrix is Toeplitz, has a positive semidefinite completion, and has pattern $P_2$.  Let $U  := \{1,\dotso,(r+1)k-1\} \setminus P_2$.  The elements of $U$ correspond to the unspecified diagonals of the submatrix.  By Corollary~\ref{cor:sdp}, there exists a Toeplitz completion for this submatrix that assigns the value $a_i$ for every $i \in U$.  Now $U$ is a subset of the unspecified diagonals of $\TT$.  We assign the value $a_i$ to the unspecified diagonals of $\TT$ for every $i \in U$ thereby obtaining a new partial positive semidefinite Toeplitz matrix, say $\TT'$, with pattern $\{1,2,\dotso,(r+1)k\}$.  Since $\TT$ and $\TT'$ agree on the diagonals of $P_2'$, every completion of $\TT'$ is also a completion of $\TT$.  The pattern of $\TT'$ is of the form $P_1$, hence it admits a positive semidefinite Toeplitz completion, which is also a completion of $\TT$, as desired.
\end{proof}



Whether or not Theorem \ref{thm:main2} gives a full characterization of
all patterns for where there is always a Toeplitz completion among all positive semidefinite completions, is an open question.

\section{The Singularity Degree of Some Toeplitz Cycles}
\label{sec:sd}
The \emph{Slater condition} holds for the feasible set of an SDP if it
contains a positive definite matrix.  If the Slater condition does not
hold for an SDP then there is no guarantee of convergence to an optimal
solution using any known algorithm, moreover, it may not be possible to
verify if a given matrix is optimal or not.  One way to regularize an
SDP that does not satisfy the Slater condition is by restricting the
problem to the smallest face of $\Snp$ containing the feasible set.
Since every face of $\Snp$ is a smaller dimensional positive
semidefinite cone, every SDP may be transformed into an equivalent
(possibly smaller dimensional) SDP for which the Slater condition holds.
This transformation is referred to as \emph{facial reduction}, see for
instance \cite{bw1,bw2,DrusWolk:16}.  The challenge, of course, is to
obtain the smallest face.  Most facial reduction algorithms look for
\emph{exposing vectors}, i.e.,~non-zero, positive semidefinite matrices that are orthogonal to the minimal face.  Exposing vectors are guaranteed to exist by the following theorem of the alternative.  Here we let $\FF$ be the feasible set of an SDP that is defined by the affine equation $\A(X) = b$, as in \eqref{eq:feasibleset}.
\begin{thm}[{\cite{bw3}}]
\label{thm:fr}
Exactly one of the following holds.
\begin{enumerate}
\item $\FF \cap \Snpp \ne \emptyset$.
\item  \label{item:expZ}
There exists $Z \in \Snp \cap \range (\A^*)$ such that $ZX = 0$ for all $X \in \FF$.
\end{enumerate}
\end{thm}
This result guarantees the existence of exposing vectors when the Slater
condition does not hold.  By restricting the feasible set of an SDP to
the kernel of an exposing vector, the dimension of the SDP is reduced.
By repeatedly finding exposing vectors and reducing the size of the SDP,
eventually the problem is reduced to the minimal face and the Slater
condition holds.  If the exposing vector obtained at each iteration is
as in Item \ref{item:expZ} of Theorem~\ref{thm:fr} and of maximal rank over all such exposing vectors, then the number of times the original SDP needs to be reduced in order to obtain a regularized SDP is referred to as the \emph{singularity degree}.  
This notion and the connection to error bounds for SDP was introduced in 
\cite[Sect. 4]{S98lmi}.  For instance, if an SDP satisfies the Slater
condition, then it has singularity degree $0$ and the singularity degree is $1$ if and only if there exists an exposing vector $Z \in \Snp \cap \range(\A^*)$ such that $\rank(Z) + \rank(X) = n$ for all $X$ in the relative interior of $\FF$.

In \cite[Lemma~3.4]{tanigawa} it is shown that for $n\ge 4$, there exists a partial matrix (not Toeplitz) with all entries of the diagonals $0,1,n-1$ specified so that the singularity degree of the corresponding SDP is at least 2.  Here we apply the results of the previous sections to derive the singularity degree (or bounds for it) of a family of symmetric partial Toeplitz matrices with pattern $P = \{1,n-1\}$.  As in much of the matrix completion literature the partial matrix is viewed as arising from a graph and the pattern $P$ corresponds to the graph of a \emph{cycle} with loops.  The following result is useful throughout.
\begin{prop} 
\label{Tinverse}
Let $T=(t_{i-j})_{i,j=1}^n$ be a positive definite Toeplitz matrix, and suppose that $(T^{-1})_{k,1}=0$ for all $k\in \{3,\ldots , n-1\}$. Then $T^{-1}$ has the form 
\begin{equation}\label{CE} \begin{bmatrix}
a & c & 0& & d \\
c & b & c & \ddots & \\
0& c & b & \ddots &0 \\
& \ddots & \ddots & \ddots & c \\
d & & 0& c & a
\end{bmatrix},\end{equation}
with $b=\frac{1}{a} (a^2+c^2-d^2)$.
\end{prop}

\begin{proof} 
Let us denote the first column of $T$ by $\begin{bmatrix} a & c & 0 &
\cdots & 0 & d \end{bmatrix}^T$. By the Gohberg-Semencul
formula we have that
$$ T^{-1} =\frac{1}{a} ( AA^T-B^TB ), $$
where 
$$ A=\begin{bmatrix}
a & 0 & 0& & 0 \\
c & a & 0 & \ddots & \\
0& c & a & \ddots &0 \\
& \ddots & \ddots & \ddots & 0 \\
d & & 0& c & a
\end{bmatrix}, B= \begin{bmatrix}
0 & d & 0& & c \\
0 & 0 & d & \ddots & \\
0& 0 & 0 & \ddots &0 \\
& \ddots & \ddots & \ddots & d \\
0& & 0& 0 & 0
\end{bmatrix}.$$
\end{proof}

\begin{ex}
\label{ex:4n}
Let $n=4$ and consider the partial matrix with pattern $P = \{1,3\}$ and data $\{t_0,t_1,t_3\} = \{1+\alpha,\cos(\frac{\theta}{3}), \cos(\theta)\}$ for $\theta \in [0,\pi]$ and $\alpha \ge0$.  Let $\FF(\alpha)$ denote the set of positive semidefinite completions for each $\alpha > 0$ as in Section~\ref{sec:sdp}, let $\FF = \FF(0)$, and let $\sd(\FF)$ denote the singularity degree of any SDP for which $\FF$ is the feasible set.  By Corollary~6 of \cite{MR1236734} there exists a positive definite completion whenever $\alpha > 0$ and there exists a positive semidefinite completion (but not a positive definite one) when $\alpha = 0$.  Then by Theorem~\ref{thm:main} the maximum determinant completion is Toeplitz whenever $\alpha  > 0$ and there exists a maximum rank positive semidefinite completion that is Toeplitz when $\alpha = 0$ by Corollary~\ref{cor:sdp}.  Let $X(\alpha)$ denote the maximum determinant positive definite completion when $\alpha > 0$.  Then 
\[
X(\alpha) =: \begin{pmatrix} 1 + \alpha &   \cos(\frac{\theta}{3}) & x(\alpha) & \cos(\theta) \cr  
\cos(\frac{\theta}{3}) & 1+\alpha  & \cos(\frac{\theta}{3}) & x(\alpha)  \cr 
x(\alpha)  &  \cos(\frac{\theta}{3}) & 1+\alpha &  \cos(\frac{\theta}{3}) \cr 
\cos(\theta) & x(\alpha)  &  \cos(\frac{\theta}{3}) & 1+\alpha \end{pmatrix}.
\] 
Here $x(\alpha)$ denotes the value of the unspecified entry.  Using the symbolic package in MATLAB, we obtain 
\[
x(\alpha) = \frac 12 \left( \sqrt{ \alpha(\alpha+2) + (4\cos^2(\frac{\theta}{3}) - 1)^2 } - (1+\alpha) \right).
\] 
Taking the limit as $\alpha$ decreases to $0$, we get
\[
\bar{X} := \lim_{\alpha \searrow 0} X(\alpha) = \begin{pmatrix} 1 &   \cos(\frac{\theta}{3}) & \cos(\frac{2\theta}{3}) & \cos(\theta) \cr  
\cos(\frac{\theta}{3}) & 1  & \cos(\frac{\theta}{3}) & \cos(\frac{2\theta}{3})  \cr 
\cos(\frac{2\theta}{3})  &  \cos(\frac{\theta}{3}) & 1 &  \cos(\frac{\theta}{3}) \cr 
\cos(\theta) & \cos(\frac{2\theta}{3})  &  \cos(\frac{\theta}{3}) & 1 \end{pmatrix}.
\]
This matrix has maximum rank over all positive semidefinite completions when $\alpha = 0$ due to Corollary~\ref{cor:sdp}.  Specifically, $\bar{X}$ has rank 2 whenever $\theta \in (0,\pi]$ and rank 1 when $\theta = 0$.
To derive the singularity degree of $\FF$ we need to find the maximal rank of an exposing vector having the properties of Theorem~\ref{thm:fr}.  To this end let $Z(\alpha) := \alpha X(\alpha)^{-1}$ and let $\bar{Z} = \lim_{\alpha \searrow 0} Z(\alpha)$.  By Theorem~\ref{thm:Xalpha}, $\bar{Z}$ exists and is an exposing vector for $\FF$ (as long as it is not the zero matrix) as in Theorem~\ref{thm:fr}.  By Proposition~\ref{Tinverse} we have
\begin{equation}
\label{eq:Zalpha}
Z(\alpha) =: \begin{pmatrix} a(\alpha) &   c(\alpha) & 0 & d(\alpha) \cr  
c(\alpha) & b(\alpha)  & c(\alpha) & 0 \cr 
0 &  c(\alpha) & b(\alpha) &  c(\alpha) \cr 
d(\alpha) & 0 &  c(\alpha) & a(\alpha) \end{pmatrix}, 
\end{equation}
where $b(\alpha) = \frac{1}{a(\alpha)}(a(\alpha)^2 + c(\alpha)^2 - d(\alpha)^2)$.  Let $a,b,c,$ and $d$ be the limit points of $a(\alpha),b(\alpha),c(\alpha),$ and $d(\alpha)$ respectively, as $\alpha$ decreases to $0$.  Then 
\[
\bar{Z} = \begin{pmatrix} a &   c & 0 & d \cr  
c & b  & c & 0 \cr 
0 &  c & b &  c \cr 
d & 0 &  c & a \end{pmatrix}.
\]
We observe that if $b \ne 0$, then $\rank(\bar{Z}) \ge 2$ and if $b = 0$ then $\rank(\bar{Z}) \le 1$.  The first observation is trivial, while for the second observation, suppose $\bar{Z} \ne 0$ from which we get that $a  > 0$.  Then since $\bar{Z}$ is positive semidefinite, we have $c = 0$ and from the equation $b = \frac{1}{a}(a^2 + c^2 - d^2)$ we get
\[
0 = \frac{1}{a}(a^2 -d^2),
\]
which implies that $a = d$ and $\rank(\bar{Z})= 1$.  Now since $X(\alpha)Z(\alpha) = \alpha I$ we have
\[
X(\alpha) \begin{bmatrix}
c(\alpha) \\
b(\alpha) \\
c(\alpha) \\
0
\end{bmatrix} = \alpha \begin{bmatrix}
0 \\
\alpha \\
0 \\
0
\end{bmatrix}.
\]
Solving for $b(\alpha)$ we obtain the expression
\[
b(\alpha) = \begin{cases} 
\frac{\alpha(\cos(\frac{\theta}{3}) + \cos(\theta))}{(1+\alpha)(\cos(\frac{\theta}{3}) + \cos(\theta))-2x(\alpha)\cos(\frac{\theta}{3})}, \quad &\theta \in [0,\pi], \ \theta \ne \frac{3\pi}{4}, \\
\frac{\alpha(1+\alpha+x(\alpha))}{(1+\alpha)(1+\alpha +x(\alpha)) - 2\cos^2(\theta)}, &\theta = \frac{3\pi}{4}.
\end{cases} 
\]
Evaluating the limits we get that $b = 0$ if $\theta = \pi$ and $b$ is non-zero for all other values of $\theta$ in $[0,\pi]$.  It follows that $\rank(\bar{Z}) \ge 2$ when $\theta \in (0,\pi)$ and since $\rank(\bar{X}) =2$ for these values of $\theta$, we conclude that $\sd(\FF) = 1$ when $\theta \in (0,\pi)$.  When $\theta = 0$ it can be derived that $a = b= \frac 34$ and $c = d = -\frac 38$.  Then $\bar{Z}$ is a rank 3 matrix and $\sd(\FF) = 1$.  For the case $\theta = \pi$ we have that $\rank(\bar{Z}) \le 1$ and now we show that every exposing vector for $\FF$ that lies in $\Snp \cap \range( \A^*)$ has rank at most 1.  Indeed, for $\theta = \pi$ we have 
\[
\bar{X} = \begin{pmatrix} 1 &   \frac12 & - \frac12 & -1 \cr  \frac12 & 1  & \frac12 & - \frac12 \cr 
- \frac12 &  \frac12 & 1 &  \frac12 \cr -1 & - \frac12 &  \frac12 & 1 \end{pmatrix} .
\]
Now a basis for the kernel of $\bar{X}$ is formed by the vectors
\[
v := \begin{bmatrix}
1 \\ 0 \\ 0 \\ -1
\end{bmatrix}, \ u := \begin{bmatrix}
1 \\ -1 \\ 1 \\0
\end{bmatrix}.
\]
Observe that $\range(\A^*)$ consists of all the matrices with entries $(1,3)$ and $(2,4)$ identically $0$.  Now if $Z$ is any exposing vector for $\FF$, we have $\bar{X}Z = 0$ and hence $Z = \lambda (vv^T) + \mu (uu^T)$ for some $\lambda, \mu \in \R$.  But since $uu^T \notin \range(\A^*)$, it follows that $\rank(Z) \le 1$ and $\sd(\FF) \ge 2$.  We conclude this example by summarizing our observations:
\[
\sd(\FF) = \begin{cases}
1, \quad & \theta \in [0,\pi),\\
\ge 2, &\theta =\pi.
\end{cases}
\]
\end{ex}

Some of the observations of this example extend to general $n \ge 4$.  First we show that the partial matrix admits a unique positive semidefinite completion, which is Toeplitz.

\begin{prop}
\label{prop:cosn}
Consider the partial symmetric $n\times n$ Toeplitz matrix with pattern $P=\{ 1, n-1\}$ and data 
\[
\{ t_0, t_1 , t_{n-1} \} = \{ 1, \cos \theta , \cos((n-1)\theta) \} , 
\] 
where $\theta \le \frac{\pi}{n-1}$.
Then the unique positive semidefinite completion is 
$$ \left( \cos \left((i-j)\theta \right) \right)_{i,j=0}^{n-1} = B^T B  , $$ where 
$$ B= \begin{pmatrix} 1 & \cos \theta & \cos (2\theta )& \cdots & \cos ((n-1)\theta) \cr 0 
& \sin\theta& \sin (2\theta ) & \cdots & \sin ((n-1)\theta) \end{pmatrix}. $$
\end{prop}

\begin{proof} Let us denote the first column of a positive semidefinite completion by $\begin{pmatrix} \cos \theta_0 & \cos \theta_1 &  \cos \theta_2 & \cdots &  \cos \theta_{n-1} \end{pmatrix}^T$, where $\theta_0 = 0 , \theta_1 =\theta$, $\theta_{n-1} = (n-1)\theta$ and $\theta_2, \ldots , \theta_{n-2} \in [0,\pi ]$. If we look at the principal submatrix in rows and columns 1, $n-1$ and $n$, we get the positive semidefinite submatrix
$$ \begin{pmatrix} 1 & \cos (\theta_{n-2} ) & \cos((n-1)\theta ) \cr \cos( \theta_{n-2} ) & 1 & \cos (\theta ) \cr 
 \cos((n-1)\theta )  & \cos(\theta) & 1 \end{pmatrix} . $$ \cite[Proposition 2]{MR1236734} yields that $ (n-1)\theta \le \theta_{n-2} +\theta$. Thus 
 \begin{equation}\label{co} \theta_{n-2} \ge (n-2)\theta . \end{equation} 
Next, consider the $(n-1)\times (n-1)$ upper left corner with data 
\[
\{ t_0, t_1 , t_{n-2} \} =\{ 1, \cos \theta , \cos \theta_{n-2} \}.
\]
By \cite[Corollary 2]{MR1236734} we have that 
\begin{equation}\label{cor2} 2 \max \{ \theta_{n-2} , \theta \} \le  (n-2) \theta + \theta_{n-2}. \end{equation}
This implies that \begin{equation}\label{cor2a}\theta_{n-2} \le (n-2)\theta. \end{equation} Combining this with 
\eqref{co} we have $\theta_{n-2} = (n-2)\theta$. If instead we looked at the principal submatrix in rows and columns 1,2, and $m$ and combine it with the $(n-1)\times (n-1)$ lower right corner, we obtain that also in the $(n,2)$th position we necessarily have $\cos ((n-2)\theta )$. Thus along the $(n-2)$th diagonal the value is $\cos ((n-2)\theta )$.

One can repeat this argument for smaller matrices (or invoke induction) and obtain that in the $k$th diagonal necessarily all entries equal $\cos \theta_k=\cos (k\theta )$, $k=2,\ldots , n-2$.
%
\end{proof}

Now we show that in case $\theta = \frac{\pi}{n-1}$, we have $\rank(\bar{X}) + \rank(\bar{Z}) \le 3 < n$ for all $n \ge 4$.

\begin{ex}
\label{ex:generaln}
Let $n\ge 4$ and consider the $n\times n$ symmetric partial Toeplitz matrix with pattern $P=\{ 1, n-1\}$ and data $\{ t_0, t_1 , t_{n-1} \} = \{ 1+\alpha , \cos (\frac{\pi}{n-1}) , -1 \}$ where $\alpha > 0$.  As in Example~\ref{ex:4n}, we let $\FF$ denote the set of positive semidefinite completions when $\alpha = 0$ and we let $X(\alpha)$ denote the maximum determinant completion when $\alpha > 0$.  By Proposition~\ref{prop:cosn}, $\FF$ is the rank $2$ matrix
\[
\bar{X} := \left( \cos \left(\frac{(i-j)\pi}{n-1}\right) \right)_{i,j=0}^{n-1},
\]
and by Theorem~\ref{thm:Xalpha}, $\bar{X} = \lim_{\alpha \searrow 0} X(\alpha)$.  If $Z(\alpha) = \alpha X(\alpha)^{-1}$ and $\bar{Z}$ is the limit of $Z(\alpha)$ as $\alpha$ decreases to $0$, we show that $\rank(\bar{Z}) \le 1$.  Let $a,b,c,$ and $d$ be the limit points of $a(\alpha),b(\alpha),c(\alpha),$ and $d(\alpha)$ respectively.  By Proposition~\ref{Tinverse}, $Z(\alpha)$ is as in \eqref{eq:Zalpha}.  We claim that if $a =0$ then $\bar{Z} = 0$.  Indeed, by the fact that $\bar{Z}$ is positive semidefinite we have $c=d=0$.  Moreover, 
\[
0 = {\rm tr} (\bar{X}\bar{Z} ) = (n-2)b,
\]
which implies that $b = 0$ and consequently $\bar{Z} = 0$.  Thus we may assume $a>0$ and the equation
\[
b = \frac 1a (a^2 + c^2 - d^2),
\]
holds.  From $ \bar{X} \bar{Z} = 0$ and the above equation, we obtain
$$ 2\cos \left(\frac{\pi}{n-1}\right)c + b = 0 ,\ a + \cos \left(\frac{\pi}{n-1}\right)c -d = 0.$$
This gives $c=-b/(2\cos(\frac{\pi}{n-1})), \ d=a -\frac{b}{2}$, and thus
$$ b=a + \frac{1}{a} \frac{b^2}{4\cos^2(\frac{\pi}{n-1})} -\frac{1}{a} \left(a-\frac{b}{2}\right)^2.$$ 
After rearranging, we obtain
$$ \frac{b^2}{a} \left(\frac14 - \frac{1}{4\cos^2(\frac{\pi}{n-1})} \right)=0 . $$
Consequently $b=0$, and $\rank( \bar{Z}) \le 1$ follows.  

Numerical experiments suggest that $\bar{Z}$ is the rank 1 matrix with $(\bar{Z})_{11}=(\bar{Z})_{nn}=(\bar{Z})_{1n}=(\bar{Z})_{n1}= \frac{n-1}{4}$ and all other entries equal to 0.


\end{ex}
\def\cprime{$'$} \def\cprime{$'$} \def\cprime{$'$}
  \def\udot#1{\ifmmode\oalign{$#1$\crcr\hidewidth.\hidewidth
  }\else\oalign{#1\crcr\hidewidth.\hidewidth}\fi} \def\cprime{$'$}
  \def\cprime{$'$} \def\cprime{$'$} \def\cprime{$'$}

\end{document}